\documentclass [a4paper,12pt]{amsart}
\usepackage{graphicx}
\usepackage[ansinew]{inputenc}    
\usepackage[all]{xy}

\usepackage{amssymb,amscd}

\newtheorem{thm}{Theorem}[section]

\def\C{\mathbb C}

\def\dim{\operatorname{dim}}

\usepackage{graphicx}
\usepackage[ansinew]{inputenc}    
\usepackage[all]{xy}
\usepackage{amssymb,amscd}

\newtheorem{cor}[thm]{Corollary}
\newtheorem{teo}[thm]{Theorem}
\newtheorem{lem}[thm]{Lemma}

\theoremstyle{definition}

\newtheorem{defi}[thm]{Definition}

\def\C{\mathbb C}

\def\dim{\operatorname{dim}}

\def\Eu{\operatorname{Eu}}
\usepackage{color}

\thanks{The first author has been partially supported by the MICINN Grant
PGC2018-094889-B-I00. 
The second author has been partially supported by FAPESP  2016/25730-0. The third author has been partially supported by CAPES. The fourth author is partially supported by CNPq Grant 309086/2017-5  and FAPESP Grant 2018/22090-5.}

\keywords{isolated hypersurface singularity, Bruce-Roberts number, logarithmic characteristic variety}

\subjclass[2000]{Primary 32S25; Secondary 58K40, 32S50}

 \everymath{\displaystyle}
\begin{document}

\title[The Bruce-Roberts number of a function on a hypersurface]{The Bruce-Roberts number of a function on a hypersurface with isolated singularity}

\author{J.J. Nuño-Ballesteros, B. Oréfice-Okamoto, B. K. L. Pereira, J.N. Tomazella}

\address{Departament de Matem\`atiques,
Universitat de Val\`encia, Campus de Burjassot, 46100 Burjassot
SPAIN}
\email{Juan.Nuno@uv.es}

\address{Departamento de Matem\'atica, Universidade Federal de S\~ao Carlos, Caixa Postal 676,
	13560-905, S\~ao Carlos, SP, BRAZIL}
\email{bruna@dm.ufscar.br}

\address{Departamento de Matem\'atica, Universidade Federal de S\~ao Carlos, Caixa Postal 676,
	13560-905, S\~ao Carlos, SP, BRAZIL}
\email{barbarapereira@dm.ufscar.br}

\address{Departamento de Matem\'atica, Universidade Federal de S\~ao Carlos, Caixa Postal 676,
	13560-905, S\~ao Carlos, SP, BRAZIL}

\email{tomazella@dm.ufscar.br}

\maketitle

%


\begin{abstract}
Let $(X,0)$ be an isolated hypersurface singularity defined by $\phi\colon(\C^n,0)\to(\C,0)$ and $f\colon(\C^n,0)\to\C$ such that the Bruce-Roberts number $\mu_{BR}(f,X)$ is finite. We first prove that $\mu_{BR}(f,X)=\mu(f)+\mu(\phi,f)+\mu(X,0)-\tau(X,0)$, where $\mu$ and $\tau$ are the Milnor and Tjurina numbers respectively of a function or an isolated complete intersection singularity. Second, we show that the logarithmic characteristic variety $LC(X,0)$ is Cohen-Macaulay. Both theorems generalize the results of a previous paper by some of the authors, in which the hypersurface $(X,0)$ was assumed to be weighted homogeneous.
\end{abstract}

\section{Introduction}

The Milnor number $\mu(f)$ is probably the most important invariant of a holomorphic function germ $f\colon(\C^{n},0)\to\C$. It is defined algebraically as the colength of the Jacobian ideal $Jf$, generated by the partial derivatives of $f$ in $\mathcal O_n$, the local ring of holomorphic function germs $(\C^n,0)\to\C$. The finiteness of $\mu(f)$ is equivalent to that $f$ has isolated singularity, which is also equivalent, by Mather's theorem, to that $f$ is finitely determined with respect to $\mathcal{R}$, the group of coordinate changes in the source. From the topological viewpoint, Milnor showed in \cite{Milnor} that the Milnor fibre of $f$ has the homotopy type of a wedge of $(n-1)$-spheres and that $\mu(f)$ is precisely the number of such spheres. Another important property is the conservation of the Milnor number, which implies that $\mu(f)$ is equal to the number of critical points of a Morsification of $f$.

An interesting and natural extension of this theory appears when we look at function germs defined over singular varieties. In \cite{B-R}, Bruce and Roberts studied function germs $f\colon(\C^{n},0)\to\C$ considered over a germ of analytic variety $(X,0)$ in $(\C^{n},0)$. They defined an invariant, which we call the Bruce-Roberts number of $f$ with respect to $X$, and denote it by $\mu_{BR}(f,X)$. This number is a generalization of the Milnor number in the sense that $\mu_{BR}(f,X)=\mu(f)$ when $X=\C^n$. Moreover, it has also the properties that $\mu_{BR}(f,X)$ is finite if and only if $f$ has isolated singularity over $(X,0)$ (in the stratified sense), if and only if $f$ is  finitely $\mathcal{R}_{X}$-determined ($\mathcal{R}_{X}$ is the subgroup of $\mathcal R$ of coordinate changes which preserve $X$).

Since then, the Bruce-Roberts number has been studied by many authors, see for instance  \cite{tomazellaruas,carlesruas,Nivaldo,Grego,Orefice,Tajima}. In particular, in \cite{Orefice} the first, second and fourth authors considered the case that $(X,0)$ is a weighted homogeneous isolated hypersurface singularity. The first main result of \cite{Orefice} is the equality
$$
\mu_{BR}(f,X)=\mu(f)+\mu(\phi,f),
$$
where $\phi\colon(\C^{n},0)\to(\C,0)$ is the defining equation of $(X,0)$ and $\mu(\phi,f)$ is the Milnor number of the isolated complete intersection singularity (ICIS) defined by $(\phi,f)\colon(\C^n,0)\to\C^2$ (in the sense of Hamm  \cite{hamm2,hamm1}). We remark that this result was produced during the PhD of the second author (see \cite{tesebruna}) and, at that time, it was conjectured that if $(X,0)$ is an isolated hypersurface singularity (not necessarily weighted homogeneous), then
\begin{equation}\label{main-result}
\mu_{BR}(f,X)=\mu(f)+\mu(\phi,f)+\mu(X,0)-\tau(X,0),
\end{equation} where $\tau(X,0)$ is the Tjurina number. This conjecture is the natural one since $\mu(X,0)=\tau(X,0)$ in the weighted homogeneous case.

In this paper we give a proof of \eqref{main-result}, which gives a simple way to compute the Bruce-Roberts number of any function with respect to any isolated hypersurface singularity. In fact, with our method, we only need to compute the Tjurina number and the colength of the image of the differential of $f$ over the trivial tangent vector fields to the hypersurface.  As a byproduct of our technique, we also obtain a new formula for the Tjurina number of an isolated hypersurface singularity in terms of its tangent vector fields  (see also \cite{Tajima}).  We remark that \eqref{main-result} has been also proved recently by Kourliouros \cite{Grego} by using a totally different approach.

The second main result of \cite{Orefice} was that the logarithmic characteristic variety $LC(X,0)$ is Cohen-Macaulay, again when $(X,0)$ is a weighted homogeneous isolated hypersurface singularity. This is important and has many interesting applications. In fact, this implies the conservation of the Bruce-Roberts number and  that  $\mu_{BR}(f,X)$ is equal to the number of critical points of a Morsification of $f$ on each logarithmic strata, counted multiplicity (see \cite{B-R}).  Here we also extend this result to the general case that $(X,0)$ is not necessarily weighted homogeneous. We remark that $LC(X,0)$ is never Cohen-Macaulay when $(X,0)$ has codimension $>1$, so the only interesting case is when $(X,0)$ is a hypersurface. We give some applications of this result for $\mu_{BR}$-constant families of functions.

\section{The Bruce-Roberts Number}

Let $\mathcal{O}_{n}$ be the ring of germs of analytic functions $f\colon(\C^{n},0)\to\C$, and let $(X,0)$ be a germ of hypersurface with isolated singularity in $(\C^n,0)$. Two germs $f,g\colon(\C^{n},0)\to\C$ are $\mathcal{R}_{X}$-equivalent, (respectively $C^{0}$-$\mathcal{R}_{X}$-equivalent) if there exists a germ of diffeomorphism (respectively homeomorphism) $\psi\colon(\C^{n},0)\to(\C^{n},0)$, such that $\psi(X)=X$ and $f=g\circ\psi$.

We denote by $\Theta_n$ the $\mathcal {O}_{n}$-module of germs of vector fields on $(\C^n,0)$, and by $\Theta_X$ the subset of $\Theta_{n}$ of vector fields which are tangent to $X$, that is, $$\Theta_X:=\{\xi\in\Theta_n \, ; \, dg(\xi)=\xi(g)\in I(X), \, \forall g\in I(X)\}$$ where $I(X)$ is the ideal in $\mathcal{O}_n$ consisting of germs of functions vanishing on $X$. We denote by $df(\Theta_{X})$ the ideal in ${\mathcal O}_{n}$  generated by $\xi(f):=df(\xi)$ such that $\xi\in\Theta_{X}$ and  $df$ is the differential of $f$. We consider $\Theta_{X}^{T}$ the $\mathcal{O}_{n}$-submodule of $\Theta_{X}$ generated by the trivial vector fields, that is, $\Theta_{X}^{T}$ is generated by $$\phi\frac{\partial}{\partial x_{i}},\frac{\partial\phi}{\partial x_{j}}\frac{\partial}{\partial x_{k}}-\frac{\partial\phi}{\partial x_{k}}\frac{\partial}{\partial x_{i}}, \textup{ with } i,j,k=1,...,n; k\neq j. $$
Therefore, for any function germ $f\colon(\C^{n},0)\to\C$, $df(\Theta_{X}^{T})=\langle\phi\tfrac{\partial f}{\partial x_{i}}\rangle+J(f,\phi)$, where $J(f,\phi)$ is the ideal in $\mathcal{O}_{n}$ generated by the maximal minors of the Jacobian matrix of $(f,\phi)$.

\begin{defi}
Let $f\in\mathcal O_n$. The number $$\mu_{BR}(f,X)=\dim_\C\frac{\mathcal O_n}{df(\Theta_X)}$$
is the \emph{Bruce-Roberts number} of $f$ with respect to $X$.
\end{defi}

It follows from Nakayama's Lemma and the Hilbert Nullstellenstaz that the Bruce-Roberts number of $f$ with respect to $(X,0)$ is finite if and only if the variety of zeros of the ideal $df(\Theta_{X})$ is equal to the origin or empty. Moreover, we have other important equivalences, for instance, that $\mu_{BR}(f,X)$ is finite if and only if $f$ is finitely $\mathcal{R}_{X}$-determined (see \cite{B-R}). 

 
Since $df(\Theta_{X})\subset Jf$, the Bruce-Roberts number is greater or equal to the Milnor number. In \cite{Orefice} we show that if $(X,0)$ is a weighted homogeneous hypersurface with isolated singularity and $f$ is a finitely $\mathcal{R}_{X}$-determined function germ, then $$\mu_{BR}(f,X)=\mu(f)+\mu(\phi,f),$$ where $\mu(\phi,f)$ is the Milnor number of the ICIS $(X\cap f^{-1}(0),0)$, as defined by Hamm in \cite{hamm2,hamm1}. 
The main goal of this work is to extend this formula to the general case that $(X,0)$ is hypersurface with isolated singularity (not necessarily weighted homogeneous). We use the following two lemmas.

\begin{lem}\label{7.8}\cite{B-R} Let $R$ be a local ring, $A$ a $q\times p$ matrix, with $p\geq q$ and $v$ a sequence of $p$ elements in $R$.
Let $\hat{A}$ denote the $(q-1)\times p$ matrix obtained by deleting the last row of $A$ and put $$u=(u_1,\dots ,u_q)^T=A.v^T,\quad \hat{u}=(u_1,\dots ,u_{q-1})^T=\hat{A}.v^T.$$
Denote the ideal in $R$ generated by the $q\times q$ minors of $A$ by $I_q(A)$ and define $I_{q-1}(\hat{A})$ in a similar way.
Finally let 
\begin{align*}
I'&=I_{q-1}(\hat{A})+\left\langle
u_1,\dots ,u_{q-1}\right\rangle,\\
I''&=I_q(A)+\left\langle
u_1,\dots ,u_q\right\rangle,\\
I&=I_q(A)+\left\langle
u_1,\dots ,u_{q-1}\right\rangle.
\end{align*}
Then, if $R$ is Cohen-Macaulay of dimension $p$ and $l\left(R/I\right)<\infty$, $$l\left(R/I\right)=l\left(R/{I'}\right)+l\left(R/{I''}\right).$$
\end{lem}

\begin{lem}\label{csr}
Let $(X,0)$ be an isolated hypersurface singularity in $(\C^{n},0)$ and $f\in{\mathcal O_{n}}$. Then, $\mu_{BR}(f,X)<\infty$ if and only if $$\dim_{\C}\frac{\mathcal O_{n}}{df(\Theta_{X}^{T})}<\infty.$$ 
\end{lem}

\begin{proof}
	Assume that $\mu_{BR}(f,X)<\infty$ but $\dim_{\C} \mathcal O_{n}/df(\Theta_{X}^{T})=\infty$. By Nakayama's Lemma and the Hilbert Nullstelllensatz, the variety $V(\langle\phi\tfrac{\partial f}{\partial x_{i}}\rangle+J(f,\phi))$ has dimension $\ge 1$.
Let $x\neq 0$ 
 such that $x\in V(\langle\phi\tfrac{\partial f}{\partial x_{i}}\rangle+J(f,\phi))$. In particular, 
 $$\phi(x)\tfrac{\partial f}{\partial x_{i}}(x)=0,\forall i=1,\dots,n.$$ 
Since $\mu(f)\le\mu_{BR}(f,X)<\infty$, $V(Jf)=\{0\}$. Hence $\phi(x)=0$, and therefore, $x\in V(\langle\phi\rangle+J(f,\phi))$. But this is in contradiction with \cite[Proposition 2.8]{carlesruas}. 
The converse follows because $df(\Theta_{X}^{T})\subset df(\Theta_{X}).$
\end{proof}

The following theorem is our first new characterization of the Bruce-Roberts number.

\begin{teo}\label{dfthetax}
Let $(X,0)$ be an isolated hypersurface singularity in $(\C^{n},0)$ and $f\in{\mathcal O_{n}}$ such that $\mu_{BR}(f,X)<\infty$. Then,
$$\mu_{BR}(f,X)=\mu(f)+\mu(\phi, f)+\mu(X,0)-\dim_{\C}\frac{df(\Theta_{X})}{df(\Theta_{X}^{T})}.$$
\end{teo}  

\begin{proof}
We consider the following function germs
\begin{align*}
H\colon(\C^{n}\times\C,0)&\to (\C,0),&g\colon(\C^{n}\times\C,0)&\to(\C,0)\\
(x,t)&\mapsto e^{t}\phi(x)&(x,t)&\mapsto f(x)+t^{2}
\end{align*}

 In order to apply Lemma \ref{7.8}, let  $$A=\begin{pmatrix}
\tfrac{\partial g}{\partial x_{1}}&...&\tfrac{\partial g}{\partial x_{n}}&\tfrac{\partial g}{\partial t}\vspace{0.1cm}\\
\tfrac{\partial H}{\partial x_{1}}&...&\tfrac{\partial H}{\partial x_{n}}&\tfrac{\partial H}{\partial t}\end{pmatrix}=\begin{pmatrix}
\tfrac{\partial f}{\partial x_{1}}&...&\tfrac{\partial f}{\partial x_{n}}&2t\vspace{0.1cm}\\
e^{t}\tfrac{\partial \phi}{\partial x_{1}}&...&e^{t}\tfrac{\partial \phi}{\partial x_{n}}&e^{t}\phi\end{pmatrix}$$  and $$v=(0,...,0,1)\in\C^{n+1},$$ then 
\begin{align*}
I' & =I_{1}(\widehat{A})+\langle u_{1}\rangle= Jf+\langle t\rangle,\\
I'' & =I_{2}(A)+\langle u_{1},u_{2}\rangle=\langle e^{t}(\tfrac{\partial f}{\partial x_{i}}\tfrac{\partial \phi}{\partial x_{j}}-\tfrac{\partial f}{\partial x_{j}}\tfrac{\partial \phi}{\partial x_{i}}), e^{t}(\tfrac{\partial f}{\partial x_{i}}\phi-2t\tfrac{\partial \phi}{\partial x_{i}}),2t,e^{t}\phi\rangle,\\
I & =I_{2}(A)+\langle u_{1}\rangle=\langle e^{t}(\tfrac{\partial f}{\partial x_{i}}\tfrac{\partial \phi}{\partial x_{j}}-\tfrac{\partial f}{\partial x_{j}}\tfrac{\partial \phi}{\partial x_{i}}), e^{t}(\tfrac{\partial f}{\partial x_{i}}\phi-2t\tfrac{\partial \phi}{\partial x_{i}}),2t\rangle.
\end{align*}

 By Lemma \ref{csr}, 
 $$\dim_{\C}\frac{{\mathcal O_{n+1}}}{I}=\dim_{\C}\frac{{\mathcal O_{n}}}{\displaystyle\langle\phi\tfrac{\partial f}{\partial x_{i}}\displaystyle\rangle+J(f,\phi)}=\dim_{\C}\frac{{\mathcal O_{n}}}{df(\Theta_{X}^{T})}<\infty,
 $$ 
 so we can use Lemma \ref{7.8}: 
 $$\dim_{\C}\left(\frac{{\mathcal O_{n+1}}}{I}\right)=\dim_{\C}\left(\frac{{\mathcal O_{n+1}}}{I'}\right)+\dim_{\C}\left(\frac{{\mathcal O_{n+1}}}{I''}\right).$$ 

On the other hand,
\begin{align*}
\dim_{\C}\frac{{\mathcal O_{n+1}}}{I'}&=\dim_{\C}\frac{{\mathcal O_{n}}}{Jf}=\mu(f),\\
\dim_{\C}\frac{{\mathcal O_{n+1}}}{I''}&=\dim_{\C}\frac{{\mathcal O_{n}}}{\langle \phi\rangle+ J(f,\phi)}=\mu(\phi
)+\mu(\phi, f),
\end{align*}
by the Lê-Greuel formula \cite{greuel}. Moreover,
$$
\mu_{BR}(f,X)=\dim_{\C}\frac{{\mathcal O_{n}}}{df(\Theta_{X}^{T})}-\dim_{\C}\frac{df(\Theta_{X})}{df(\Theta_{X}^{T})},$$ which gives   
$$\mu_{BR}(f,X)=\mu(f)+\mu(\phi, f)+\mu(X,0)-\dim_{\C}\frac{df(\Theta_{X})}{df(\Theta_{X}^{T})}.$$
\end{proof}

\section{The Tjurina number}

In this section we present a characterization of the Tjurina number of an isolated hypersurface singularity $(X,0)$ in terms of the tangent vectors fields to $(X,0)$.


\begin{lem}\label{lema1}
	
	Let $\phi\colon(\C^{n},0)\to(\C,0)$ be an analytic germ with isolated singularity and let $X=\phi^{-1}(0)$. Assume that $\eta\in\Theta_X$ and $d\phi(\eta)=\lambda\phi$, for some $\lambda\in \mathcal O_n$. Then $\eta\in \Theta_{X}^{T}$ if and only if $\lambda\in J\phi$.
	
\end{lem}

\begin{proof} If $\eta\in \Theta_{X}^{T}$, obviously $\lambda\in J\phi$. Conversely, if $\lambda\in J\phi$ there exists $\xi\in \Theta_n$ such that $d\phi(\xi)=\lambda$. Thus, $d\phi(\eta-\phi\xi)=0$. Since $\phi$ has isolated singularity, $\partial\phi/\partial x_1,\dots,\partial\phi/\partial x_n$ form a regular sequence in $\mathcal O_n$. Therefore, the module of syzygies of $\partial\phi/\partial x_1,\dots,\partial\phi/\partial x_n$ is generated by the trivial relations. We have that $\eta-\phi\xi\in \Theta_{X}^{T}$. But $\phi\xi$ is also in $\Theta_{X}^{T}$, so $\eta\in\Theta_{X}^{T}$.
	
\end{proof}

The following theorem shows that, surprisingly, $\dim_{\C} df(\Theta_{X})/df(\Theta_{X}^{T})$ does not depend on the function germ $f\in\mathcal{O}_{n}$, provided it is finitely $\mathcal{R}_{X}$-determined. As a consequence, we show that this dimension is equal to the Tjurina number of $(X,0)$.

\begin{teo}\label{independe de f}
Let $(X,0)\subset(\C^n,0)$ be an isolated hypersurface singularity
and let $f\in\mathcal O_n$ be finitely $\mathcal{R}_{X}$-determined. Let $E:\Theta_{X}\to df(\Theta_{X})$ be the evaluation map given by $E(\xi)=df(\xi).$ Then $E$ induces an isomorphism
$$\overline{E}:\frac{\Theta_{X}}{ \Theta_{X}^{T}}\to\frac{df(\Theta_{X})}{df(\Theta_{X}^{T})}.$$
\end{teo}
\begin{proof}
Obviously, $E$ is an epimorphism and $E( \Theta_{X}^{T})=df(\Theta_{X}^{T})$. We only need to show that $\ker E\subset  \Theta_{X}^{T}$. That is, we have to show that if $\eta\in\Theta_X$ such that $df(\eta)\equiv 0$, then $\eta\in\Theta_{X}^{T}$.

We have $\mu_{BR}(f,X)<\infty$. Let $\phi\colon(\C^{n},0)\to(\C,0)$ be an analytic germ with isolated singularity such that $X=\phi^{-1}(0)$. By \cite[Proposition 2.8]{carlesruas}, $(\phi,f)\colon(\C^{n},0)\to(\C^2,0)$ defines an ICIS. Therefore, the variety determined by $J(\phi,f)$, $V(J(\phi,f))$, has dimension 1 (see \cite{looijenga}). Moreover, since $f$ and $\phi$ have isolated singularity, the rank of $(J(\phi,f)(x))$ is equal to $1$ for $x\neq 0$ in a neighborhood of the origin. Hence, $V(J(\phi,f))$  is an isolated determinantal singularity (IDS) of dimension 1 and, therefore, it is reduced, by \cite[lemma 2.5]{not2}. 

Suppose that $\eta\notin\Theta_{X}^{T}$. By Lemma \ref{lema1}, $d\phi(\eta)=\lambda\phi$, for some $\lambda\notin J\phi$. On the other hand, $f$ has also isolated singularity, so $\partial f/\partial x_1,\dots,\partial f/\partial x_n$ form a regular sequence in $\mathcal O_n$. Therefore, the module of syzygies of $\partial f/\partial x_1,\dots,\partial f/\partial x_n$ is generated by the trivial relations. Since $df(\eta)\equiv 0$, we have that
$\eta \in\displaystyle\langle\tfrac{\partial f}{\partial x_{i}}\tfrac{\partial }{\partial x_{j}}-\tfrac{\partial f}{\partial x_{j}}\tfrac{\partial }{\partial x_{i}} \displaystyle\rangle$.

This implies that $\lambda\phi=d\phi(\eta)\in J(\phi,f)$ and hence $V(J(\phi,f))\subset V(\lambda\phi)=V(\lambda)\cup V(\phi)$.
Let $C_1,\dots,C_k$ be the irreducible components of $V(J(\phi,f))$. Since 
$V(J(\phi,f))$ is reduced of dimension 1, each $C_{i}$ is an irreducible curve in $(\C^{n},0)$. In particular, either $C_i\subset V(\lambda)$ or $C_i\subset V(\phi)$, for each $i=1,\dots,k$.

If $C_{i}\not\subset V(\phi)$ for all $i=1,...,k$, then $C_{i}\subset V(\lambda)$ for all $i=1,...,k$, hence $V(J(\phi,f))\subset V(\lambda)$, therefore $\lambda\in\sqrt{\lambda}\subset\sqrt{J(\phi,f)}=J(\phi,f)\subset J\phi$, which is a contradiction with the way $\lambda$ was chosen. Otherwise, if $C_{i}\subset V(\phi)$ for some $i$, then $C_{i}\subset V(\langle\phi\rangle+J(f,\phi))$. This gives $\dim_\C \mathcal O_n/(\langle\phi\rangle+J(f,\phi))=\infty$, in contradiction with the Lê-Greuel formula for the ICIS $V(\phi,f)$.

\end{proof}


We are ready now to characterize the Tjurina number in terms of the tangent vector fields to the hypersurface $(X,0)$.

\begin{teo}\label{tjurina}
Let $(X,0)\subset(\C^n,0)$ be an isolated hypersurface singularity and let $p:\C^{n}\to\C$ be a non-zero linear function.
Then $$\tau(X,0)=\dim_{\C}\frac{dp(\Theta_{X})}{dp(\Theta_{X}^{T})}.$$
\end{teo}
\begin{proof}

Since $p$ is a non-zero linear function, we may suppose that $p(x_{1},...,x_{n})=a_{1}x_{1}+\dots+a_{n}x_{n}$, with $a_{1}\neq 0$. We consider the following sequence
$$
0\longrightarrow \ker(\alpha)\stackrel{i}{\longrightarrow}\frac{\mathcal{O}_{n}}{dp(\Theta_{X}^{T})}\stackrel{\alpha}{\longrightarrow}\frac{\mathcal{O}_{n}}{dp(\Theta_{X}^{T})}\stackrel{\pi}{\longrightarrow}\frac{\mathcal{O}_{n}}{\langle\phi\rangle+J\phi}\longrightarrow 0$$
where  $\alpha$ is given by multiplication by $\overline{\partial \phi/\partial x_1}$, $i$ is the inclusion and $\pi$ is the projection. 

The sequence above is exact. Indeed, $dp(\Theta_{X}^{T})=\langle\phi\rangle+J(p,\phi)\subset\langle\phi\rangle+J\phi$. Moreover, 
$$\ker(\pi)=\frac{\langle\phi\rangle+J\phi}{\langle\phi\rangle+J(p,\phi)},$$ 
and   
$$\operatorname{Im}(\alpha)=\frac{\displaystyle\langle\tfrac{\partial\phi}{\partial x_{1}}\rangle\displaystyle+\langle \phi\rangle+J(p,\phi)}{\langle \phi\rangle+J(p,\phi)},$$
so $\ker(\pi)\subset \operatorname{Im}(\alpha)$. The opposite inclusion follows because $\pi\circ\alpha=0$. 
Hence, 
$$\dim_{\C}\ker(\alpha)=\dim_{\C}\frac{{\mathcal O_{n}}}{\langle\phi\rangle+J\phi}=\tau(X,0).$$
We claim that 
$$
\ker(\alpha)=\frac{dp(\Theta_{X})}{dp(\Theta_{X}^{T})}=\frac{dp(\Theta_{X})}{\langle \phi\rangle+J(p,\phi)}.
$$

In fact, given $g\in dp(\Theta_{X})$ there exists $\xi\in\Theta_{X}$ such that $g=dp(\xi)$.
We have $g=d\phi(\xi)=\displaystyle\sum_{i=1}^{n} \xi_{i}\tfrac{\partial\phi}{\partial x_{i}}=\lambda\phi$, for some $\lambda\in{\mathcal O}_{n}$, and $dp(\xi)=\displaystyle\sum_{i=1}^{n} a_{i}\xi_{i}$. Therefore,
\begin{align*}
\frac{\partial\phi}{\partial{x_{1}}} g&=\frac{\partial\phi}{\partial{x_{1}}}\displaystyle\sum_{i=1}^{n}a_{i}\xi_{i}
=a_{1}\left(\lambda\phi-\sum_{i=2}^{n}\xi_{i}\frac{\partial\phi}{\partial x_{i}}\right)+\sum_{j=2}^{n} a_{j}\xi_{j}\frac{\partial\phi}{\partial x_{1}}\\
               &=a_{1}\lambda\phi+\sum_{j=2}^{n}\xi_{j}\left(a_{j}\frac{\partial\phi}{\partial x_{1}}-a_{1}\frac{\partial\phi}{\partial x_{j}}\right)\in \langle \phi\rangle+J(p,\phi).
\end{align*}
This shows the inclusion  
$$\frac{dp(\Theta_{X})}{\langle \phi\rangle+J(p,\phi)}\subset \ker(\alpha).$$

Conversely, let $g\in \mathcal O_n$ be such that $g\frac{\partial\phi}{\partial x_{1}}\in \langle \phi\rangle+ J(p,\phi)$. There exist $b,b_{ij}\in {\mathcal O_{n}}$ such that
\begin{align*}
\frac{\partial\phi}{\partial{x_{1}}} g&=b\phi+\sum_{j<i}b_{ij}\left( a_{i}\frac{\partial\phi}{\partial x_{j}}-a_{j}\frac{\partial\phi}{\partial x_{i}}\right)\\
                                     &=b\phi+\frac{\partial\phi}{\partial x_{1}}(-b_{12}a_{2}-\dots-b_{1n}a_{n})+\frac{\partial\phi}{\partial x_{2}}( b_{12}a_{1}-b_{23}a_{3}-\dots-b_{2n}a_{n})\\
                                     &\quad+\dots+\frac{\partial\phi}{\partial{x_{n}}}( b_{1n}a_{1}+b_{2n}a_{2}+\dots+b_{(n-2)n}a_{n-2}+b_{(n-1)n}a_{n-1})\\
                                     &=b\phi+\sum_{j=1}^{n}\frac{\partial\phi}{\partial x_{j}}\left(\sum_{\substack{k=1,\\ k \neq j}}^{n}b'_{jk}a_{k}\right),
\end{align*}
 where $b'_{jk}=-b_{jk}$, if $j<k$ and $b'_{jk}=b_{kj}$, if $k<j$. Thus,
   $$b\phi=( g-\sum_{k=2}^{n}b'_{1k}a_{k})\frac{\partial\phi}{\partial x_{1}}-\sum_{j=1}^{n}\frac{\partial\phi}{\partial x_{j}}\left(\sum_{\substack{k=1,\\ k \neq j}}^{n}b'_{jk}a_{k}\right).$$
 It follows that 
 $$\xi=( g-\sum_{k=2}^{n}b'_{1k}a_{k})\frac{\partial}{\partial x_{1}}-\sum_{j=1}^{n}\frac{\partial}{\partial x_{j}}\left(\sum_{\substack{k=1,\\ k \neq j}}^{n}b'_{jk}a_{k}\right)\in\Theta_{X},$$
   and a simple calculation shows that $dp\left(\tfrac{1}{a_{1}}\xi\right)=g$. This gives the opposite inclusion 
   $$\ker(\alpha)\subset\frac{dp(\Theta_{X})}{\langle \phi\rangle + J(p,\phi)}.$$
\end{proof}

The following corollary is an immediate consequence of  Theorems \ref{independe de f} and \ref{tjurina}.

\begin{cor}\label{cor tjurina}
Let $(X,0)\subset(\C^n,0)$ be an isolated hypersurface singularity and let $f\in\mathcal{O}_{n}$ be finitely $\mathcal{R}_{X}$-determined. Then $$\dim_{\C}\frac{\Theta_{X}}{\Theta_{X}^{T}}=\dim_{\C}\frac{df(\Theta_{X})}{df(\Theta_{X}^{T})}=\tau(X,0).$$
\end{cor}

We remark that Tajima in \cite{Tajima} has proved the same result with a different approach.

\section{Main result and some applications}

Finally, we prove the main result which gives the relationship between the Milnor number and the Bruce-Roberts number of $f$ with respect to an isolated hypersurface singularity $(X,0)$.

\begin{cor}\label{principal}
Let $(X,0)\subset(\C^{n},0)$ be the an isolated hypersurface singularity defined by $\phi\colon(\C^{n},0)\to(\C,0)$ and let $f\in\mathcal O_n$ be finitely $\mathcal{R}_{X}$-determined. Then,
$$\mu_{BR}(f,X)=\mu(f)+\mu(\phi, f)+\mu(X,0)-\tau(X,0).$$
\end{cor}

\begin{proof}
It follows from Theorem \ref{dfthetax} and Corollary \ref{cor tjurina}.
\end{proof}

The following results are direct applications of Corollary \ref{principal}.

\begin{cor}
Let $f,\phi\colon(\C^{n},0)\to(\C,0)$ be function germs with isolated singularity, and let $(X,0)$ and $(Y,0)$ be the hypersurfaces determined by $\phi$ and $f$, respectively. 
If $\mu_{BR}(f,X)<\infty$ and $\mu_{BR}(\phi,Y)<\infty$, then $$\mu_{BR}(f,X)-\mu_{BR}(\phi,Y)=\tau(Y,0)-\tau(X,0).$$  
\end{cor}
\begin{proof}
The hypothesis imply that $(f,\phi)$ defines an ICIS, and $\mu(\phi,f)=\mu(f,\phi)$. By Corollary \ref{principal}, we have then $\mu_{BR}(f,X)-\mu_{BR}(\phi,Y)=\tau(Y,0)-\tau(X,0).$

\end{proof}
 
The next corollary shows that 
the Bruce-Roberts number is a topological invariant when $(X,0)$ is an isolated hypersurface singularity.

\begin{cor}\label{joao}
Let $(X,0)\subset (\C^{n},0)$ be an isolated hypersurface singularity.  Let $f,g\in\mathcal O_n$ be finitely $\mathcal{R}_{X}$-determined function germs such that $f$ is $C^{0}$-${\mathcal{R}}_{X}$-equivalent to $g$. 
Then $\mu_{BR}(f,X)=\mu_{BR}(g,X).$
\end{cor}

\begin{proof}
Let $h:(\C^n,0)\to(\C^n,0)$ be a homeomorphism such that $h(X)=X$ and $g=f\circ h$. We have $h(X\cap f^{-1}(0))=X\cap g^{-1}(0)$. If  $(X,0)$ is defined by $\phi\colon(\C^{n},0)\to(\C,0)$, then $\mu(f,\phi)=\mu(g,\phi)$, because the Milnor number of an ICIS is a topological invariant. Moreover, $f$ and $g$ are also $C^{0}$-$\mathcal{R}$-equivalent, so $\mu(f)=\mu(g)$ by \cite{Milnor}. By Corollary \ref{principal}, we get $\mu_{BR}(f,X)=\mu_{BR}(g,X)$.
\end{proof}


From Corollary \ref{principal} and the Lê-Greuel formula \cite{greuel}, we also have the following:

\begin{cor}\label{nbrsemt} 
Let $(X,0)$ be the isolated hypersurface singularity defined by $\phi\colon(\C^{n},0)\to(\C,0)$, and let $f\in\mathcal{O}_{n}$ be finitely $\mathcal{R}_{X}$-determined. Then 
$$\mu_{BR}(f,X)=\dim_\C \frac{\mathcal O_n}{\langle f\rangle+J(f,\phi)}+\mu(X,0)-\tau(X,0).$$


\end{cor}


We can use our results to relate the Bruce-Roberts number of a generic linear projection to the top polar multiplicity of an hypersurface as defined by Gaffney \cite{Gaffney}.  Let $(X,0)\subset(\C^{n},0)$ be an isolated hypersurface singularity defined as the zero set of a function germ $\phi\colon(\C^n,0)\to\C$, the ($n-1$)-th polar multiplicity is defined as
$$m_{n-1}(X,0)=\dim_{\C}\frac{\mathcal{O}_{n}}{\langle\phi\rangle+J(p,\phi)},$$
where $p\colon(\C^{n},0)\to(\C,0)$ is a generic linear projection. From the proofs of Theorems \ref{dfthetax} and  \ref{independe de f}, we observe that
$$\mu_{BR}(p,X)=\dim_{\C}\frac{\mathcal{O}_{n}}{\langle\phi\tfrac{\partial p}{\partial x_{i}}\rangle+J(p,\phi)}-\tau(X,0),$$
therefore, 
$$m_{n-1}(X,0)=\mu_{BR}(p,X)+\tau(X,0).$$

On the other hand, in \cite{amigosjoao}, Perez and Saia show that
\begin{equation*}\label{ps}
	1+(-1)^{n-1}\mu(X,0)=\sum_{i=0}^{n-1}(-1)^{i}m_{i}(X,0),
\end{equation*}
where $m_{i}(X,0)$ is the $i$-th polar multiplicity. From Lê-Tessier's formula, we have
\begin{equation*}\label{lt}
	\Eu(X,0)=\sum_{i=0}^{n-2}(-1)^{n-i-2}m_{i}(X,0),
\end{equation*}
where $\Eu(X,0)$ denotes the Euler obstruction of $(X,0)$ (see \cite{eulerobstruction}). Hence,
$$m_{n-1}(X,0)-\mu (X,0)=\Eu(X,0)+(-1)^{n-1}.$$

We can conclude, hence, the following corollary.
\begin{cor}\label{multiplicidade polar}
	Let $p\colon(\C^{n},0)\to(\C,0)$ a generic linear projection and $(X,0)\subset(\C^n,0)$ be a hypersurface with isolated singularity. Then:
	\begin{enumerate}
	 \item $m_{n-1}(X,0)=\mu_{BR}(p,X)+\tau(X,0)$;
	 \item  $\Eu(X,0)=\mu_{BR}(p,X)+\tau(X,0)-\mu(X,0)+(-1)^n.$
	 \end{enumerate}
\end{cor}

We remark, by Corollary \ref{multiplicidade polar}, that the Bruce-Roberts number of a generic linear projection on an isolated hypersurface singularity does not depend on the projection.

\section{The logarithmic characteristic variety}

We recall the definitions of logarithmic stratification and of logarithmic characteristic variety due to Saito \cite{saito}. Let $(X,0)\subset(\C^n,0)$ be a germ of a reduced analytic subvariety. Take a representative $X$ on some small open neighbourhood $U$ of 0 in $\C^{n}$. For each $x\in U$, we denote by $\Theta_{X}(x)$ the linear subspace of $T_{x}U$ generated by the vectors $\delta(x)$, with $\delta\in \Theta_{X,x}$.

\begin{lem}\cite{B-R,saito}
There is a unique stratification $\{X_{\alpha};\;\alpha \in I\}$ of $U$ that satisfies the following properties:
\begin{enumerate}
\item Each stratum $X_{\alpha}$ is a smooth connected immersed submanifold of $U$ and $U$ is the disjoint union  $\cup_{\alpha\in I}X_{\alpha}.$
\item If $x\in U$ lies in a stratum $X_{\alpha}$, then the tangent space $T_{x}X_{\alpha}$ coincides with $\Theta_{X}(x).$
\item If $X_\alpha$ and $X_\beta$ are two distinct strata with $X_\alpha$ meeting the closure of $X_\beta$ then $X_\alpha$ is contained in the frontier of $X_\beta$.
\end{enumerate}
\end{lem}

\begin{defi}
The stratification $\{X_\alpha:\alpha\in I\}$ of the previous lemma is called the \emph{logarithmic stratification} of $X$ and the strata $X_\alpha$, the \emph{logarithmic strata}. The germ $(X,0)$ is {\em holonomic} if, for some neighborhood $U$ of $0$ in $\C^n$, the logarithmic stratification has only finitely many strata.
\end{defi}

\begin{defi} The \emph{logarithmic characteristic variety}, $LC(X,0)$, is defined as follows. 
Suppose the vector fields $\delta_1,\dots ,\delta_m$ generate $\Theta_X$ on some neighborhood $U$ of $0$ in $\C^n$. Let $T^*_U\C^n$ be the restriction of the cotangent bundle of $\C^n$ to $U$. We define $LC_U(X)$ to be 
$$
LC_U(X)=\{(x,\xi)\in T^*_U\C^n:\xi(\delta_i(x))=0, i=1,\dots ,m\}.
$$ 
Then $LC(X,0)$ is the germ of $LC_U(X)$ in $T^*\C^n$ along $T^*_0\C^n$, the cotangent space to $\C^n$ at $0$.
\end{defi}

We observe that $LC(X,0)$ is a well defined germ of analytic subvariety in $T^*\C^n$ which is independent of the choice of the vector fields $\delta_i$ (see \cite{B-R,saito} for details). If $(X,0)$ is holonomic with logarithmic strata $X_0,\dots,X_k$ then $LC(X,0)$ has dimension $n$, with irreducible components $Y_0,\dots,Y_k$, where $Y_i=\overline{N^*X_i}$, the closure of the conormal bundle $N^*X_i$ of $X_i$ in $\C^n$ (see \cite[Proposition 1.14]{B-R}).

There exists a deep connection between the logarithmic characteristic variety and the Bruce-Roberts number. If $LC(X,0)$ is Cohen-Macaulay, then $\mu_{BR}(f,X)$ is equal to the number of critical points of a Morsification of $f$ on each logarithmic stratum $X_\alpha$, counted with multiplicity (see \cite[Corollary 5.8]{B-R}). However, it also is well known (see \cite[Proposition 5.8]{B-R}) that $LC(X,0)$ is never Cohen-Macaulay when $(X,0)$ has codimension $>1$. Thus, the only interesting case to look at is when $(X,0)$ is hypersurface. 

In a previous paper \cite{Orefice}, we showed that $LC(X,0)$ is Cohen-Macaulay when $(X,0)$ is a weighted homogeneous isolated hypersurface singularity. This fact had several interesting corollaries (see \cite{tomazellaruas,Nivaldoerrata,Nivaldo,Orefice}). In the next theorem, we extend this result to the general case that $(X,0)$ is an isolated hypersurface singularity (not necessarily weighted homogeneous).

%
%

\begin{teo}\label{LC(X)}
Let $(X,0)$ be any isolated hypersurface singularity. Then $LC(X,0)$ is Cohen-Macaulay.
\end{teo}

\begin{proof}
We assume that $X=\phi^{-1}(0)$, with $\phi:(\C^n,0)\to(\C,0)$. We have to prove that $LC(X,0)$ is Cohen-Macaulay at every point $(0,p)\in LC(X,0)$. Let $f\in\mathcal O_n$ be finitely $\mathcal{R}_{X}$-determined such that $df(0)=p$. Let $C_1,\dots,C_k$ be the irreducible components of $X$. The logarithmic strata of $X$ are $X_0=\C^n\smallsetminus X$, $X_i=C_i\smallsetminus\{0\}$ for $i=1,\dots,k$ and $X_{k+1}=\{0\}$. Let $F\colon(\C^n\times\C,0)\to\C$ be given by $F(x,t)=f_t(x)$ a Morsification of $f$. It follows from \cite[Corollary 5.8]{B-R} that $LC(X,0)$ is Cohen-Macaulay at $(0,p)$ if and only if 
\[
\mu_{BR}(f,X)=\sum_{i=0}^{k+1} n_i m_i,
\] 
where $n_i$ is the number of critical points of $f_t$ on the stratum $X_i$ and $m_i$ is the multiplicity of the corresponding irreducible component $Y_i$ of $LC(X,0)$. 

By Corollary \ref{nbrsemt}, we know that
\[
\mu_{BR}(f,X)=\dim_\C \frac{\mathcal O_n}{\langle f\rangle+J(f,\phi)}+\mu(X,0)-\tau(X,0).
\]
We write $\mathcal O_{n+1}/(\langle F\rangle+J(f_t,\phi))=R/I$, where $R=\mathcal O_{n+1}/\langle F\rangle$ and $I=(\langle F\rangle+J(\phi,f_t))/\langle F\rangle$. The ring $R$ is Cohen-Macaulay of dimension $n$ and $I$ is generated by the 2-minors of a matrix of size $n\times 2$. Since $\dim R/I=1$ and $1=n-(n-2+1)(2-2+1)$, it follows that $R/I$ is determinantal and, therefore, Cohen-Macaulay, by Eagon-Hochster result \cite{E-H}. In particular, we have conservation of multiplicity, that is, for all $t\ne0$, we have
\[
\dim_\C \frac{\mathcal O_n}{\langle f\rangle+J(f,\phi)}=\sum_{i=0}^{k+1}\sum_{x\in X_i}\dim_\C \frac{\mathcal O_{n,x}}{\langle f_t\rangle+J(f_t,\phi)}.
\]
We remark the sum in the right hand side has only a finite number of terms, corresponding to the critical points of $f_t$ on $X_i$. 

When $i=1,\dots,k$ and $x\in X_i$, $X$ is smooth at $x$ and $m_i=1$, hence  
\[
\sum_{x\in X_i}\dim_\C \frac{\mathcal O_{n,x}}{\langle f_t\rangle+J(f_t,\phi)}=\sum_{x\in X_i}\mu_{BR}(f_t,X)_x=\sum_{i=1}^{k}n_i=\sum_{i=1}^{k}n_im_i,
\]
since $f_t$ is a Morse function (see \cite[Proposition 5.12]{B-R}). For $i=0$, if $x\in X_0$ then $x\notin X$. We have
\[
\sum_{x\in X_0}\dim_\C \frac{\mathcal O_{n,x}}{\langle f_t\rangle+J(f_t,\phi)}=\sum_{x\in X_0}\mu(f_t)_x=n_0=n_0m_0,
\]
since $f_t$ is a Morse function and in this case $m_0=1$. Finally, for $i=k+1$ we have just one critical point $x=0$ in $X_{k+1}=\{0\}$. Thus,
\[
\dim_\C \frac{\mathcal O_{n,0}}{\langle f_t\rangle+J(f_t,\phi)}=\mu_{BR}(f_t,X)_0-\mu(X,0)+\tau(X,0)=n_{k+1}m_{k+1}-\mu(X,0)+\tau(X,0),
\]
since $f_t$ is a Morse function and in this case $n_{k+1}=1$. Summing up for all $i=0,\dots,k+1$ we get
\[
\dim_\C \frac{\mathcal O_n}{\langle f\rangle+J(f,\phi)}=\sum_{i=0}^{k+1}n_im_i-\mu(X,0)+\tau(X,0)
\]
and hence $\mu_{BR}(f,X)=\sum_{i=0}^{k+1}n_im_i$.

\end{proof}

\begin{cor}\label{similar bruna1} Let $(X,0)$ be an isolated hypersurface singularity and let $f\in\mathcal{O}_{n}$ be finitely $\mathcal{R}_{X}$-determined. Then,
$$\mu_{BR}(f,X)=\mu(f)+N+m_{n-1}(X,0)-\tau(X,0),$$ where $N$ is the number of critical points of a Morsification of $f$ on $X\smallsetminus\{0\}$.
\end{cor}

\begin{proof}
	By Theorem \ref{LC(X)} and \cite[Corollary 5.8, Propositions 5.12 and 5.14]{B-R}, we have
\begin{align*}
\mu_{BR}(f,X)&=\sum^{k+1}_{i=0}n_{i}m_{i}\\
             &=\mu(f)+N+n_{k+1}m_{k+1}\\
             &=\mu(f)+N+\mu_{BR}(p,X)_0\\
             &\stackrel{(*)}{=}\mu(f)+N+\dim_{\C}\frac{{\mathcal O}_{n}}{\langle p\rangle+J(p,\phi)}+\mu(X,0)-\tau(X,0)\\
             &\stackrel{(**)}{=}\mu(f)+N+\mu(p)+\mu(X\cap p^{-1}(0),0)+\mu(X,0)-\tau(X,0)\\
             &=\mu(f)+N+m_{n-1}(X,0)-\tau(X,0),\\
\end{align*}
where $p:(\C^{n},0)\to(\C,0)$ is a generic linear projection, $(*)$ follows from Corollary \ref{nbrsemt}, $(**)$ follows from the Lê-Greuel formula and the last equality follows again from Lê-Greuel formula and from the definition of the polar multiplicity.
\end{proof}

The constancy of the Milnor number in a family of holomorphic functions is controlled by means of the integral closure of the Jacobian ideal (see \cite{greuel2}). Similar results have been obtained for the Bruce-Roberts number in \cite{tomazellaruas}, but there the authors need the additional hypothesis that $LC(X,0)$ is Cohen-Macaulay. It follows from our Theorem \ref{LC(X)} that the results in \cite{tomazellaruas} are true for any isolated hypersurface singularity.


More specifically,  let $(X,0)\subset(\C^n,0)$ be an isolated hypersurface singularity and let $f\colon(\C^n,0)\to \C$ be finitely $\mathcal R_X$-determined. Given
$F\colon(\C^{n}\times\C, 0)\to(\C, 0)$ a deformation of $f$, we put $f_t(x)=F(x,t)$.
We say that $F$ is a \emph{$\mu_{BR}$-constant} deformation of $f$ if $\mu_{BR}(f_t,X)=\mu_{BR}(f,X)$ for $t$ small enough. 

We also recall that the \emph{polar curve} of $F$ is defined as
$$C:=\{(x,t)\in \C^{n}\times\C;\; dF(\xi_{i})(x,t) = 0\, \forall i = 1,\dots,p\},$$ 
where $\xi_1,\dots,\xi_p$ are generators of $\Theta_X$. The following corollary is an immediate consequence of the results of \cite{tomazellaruas} and Theorem \ref{LC(X)}.

\begin{cor}\label{cor-tomazellaruas}
With the above notation, the following statements are equivalent:
	\begin{itemize}
		\item [(i)] $F$ is a $\mu_{BR}$-constant deformation of $f$;
		\item[(ii)] The polar curve of $F$ with respect to $\{t=0\}$ does not split, that is, $C=\{0\}\times\C.$
	\end{itemize}
Moreover, if $\frac{\partial F}{\partial t}\in\overline{df(\Theta_{X})}$ then $F$ is a $\mu_{BR}$-constant deformation of $f$ (where $\overline{I}$ is the integral closure of an ideal $I$).
		
\end{cor}

Other interesting results for $\mu_{BR}$-constant families of functions are obtained by Grulha in \cite{Nivaldo}. Again, these results need the additional hypothesis that $LC(X,0)$ is Cohen-Macaulay, which is not necessary any more by our Theorem \ref{LC(X)}. In the following corollary we denote by $\Eu(f,X)$ the relative local Euler obstruction of a function $f$ on $(X,0)$ (see \cite{Nivaldo} for details).

\begin{cor} With the same notation as in Corollary \ref{cor-tomazellaruas}:
	\begin{itemize}
		\item[(i)] If $\mu_{BR}(f_t,X)$ is constant, then $\mu(f_{t})$, $\mu(X\cap f_t^{-1}(0),0)$ and $\Eu(f_{t},X)$ are constant in the family.		
		\item[(ii)] If $\mu(f_{t})$ is constant, then the constancy of $\Eu(f_{t},X)$ or $\mu(X\cap f_t^{-1}(0),0)$ implies that $\mu_{BR}(f_t,X)$ is constant in the family,
	\end{itemize}
\end{cor}

\printindex
\end{document}